\documentclass[11pt, reqno]{amsart}
\usepackage{amsfonts}
\usepackage{amsmath}
\usepackage{amssymb}
\usepackage{amscd}
\usepackage[mathscr]{eucal}
\usepackage{url}

\allowdisplaybreaks[1]

\newcommand{\ph}[2]{{\left({#1}\right)}_{#2}}

\newcommand{\gf}[1]{\Gamma{\left({#1}\right)}}

\renewcommand*{\bar}{\overline}

\newcommand{\bin}[2]{\left({\genfrac{}{}{0pt}{}{#1}{#2}}\right)}

\theoremstyle{plain}

\newtheorem{theorem}{Theorem}[section]

\newtheorem{prop}[theorem]{Proposition}
\newtheorem{cor}[theorem]{Corollary}
\theoremstyle{definition}
\newtheorem{defi}[theorem]{Definition}
\newtheorem{remark}[theorem]{Remark}

\numberwithin{equation}{section}

\begin{document}

\title[Transformations of Well-Poised Hypergeometric Functions over Finite Fields]{Transformations of Well-Poised Hypergeometric Functions over Finite Fields}
\author{Dermot M\lowercase{c}Carthy}  

\address{Department of Mathematics, Texas A\&M University, College Station, TX 77843-3368, USA}

\email{mccarthy@math.tamu.edu}




\begin{abstract}
We define a hypergeometric function over finite fields which is an analogue of the classical generalized hypergeometric series. 
We prove that this function satisfies many transformation and summation formulas. 
Some of these results are analogous to those given by Dixon, Kummer and Whipple for the well-poised classical series.
We also discuss this function's relationship to other finite field analogues of the classical series, most notably those defined by Greene and Katz. 
\end{abstract}

\maketitle


\section{Introduction and Statement of Results}\label{sec_Intro}
The main goal of this paper is to find analogues of classical generalized hypergeometric series transformations, particularly Whipple's results on well-poised series, for hypergeometric functions over finite fields. 
Hypergeometric functions over finite fields have appeared in various forms in the literature (for example \cite{G}, \cite{Ka}) and our motivation for this work is their links to Fourier coefficients of certain modular forms \cite{A, AO, E, F2, DMC2, M, P} 
and the expectation that these transformations will lead to new identities between Fourier coefficients of modular forms. 
(This will be the subject of a forthcoming paper by the author and Matt Papanikolas.)

We start by recalling the classical generalized hypergeometric series ${_rF_s}[\cdots]$. For a complex number $a$ and a non-negative integer $n$ let $\ph{a}{n}$ denote the rising factorial defined by
\begin{equation*}\label{RisFact}
\ph{a}{0}:=1 \quad \textup{and} \quad \ph{a}{n} := a(a+1)(a+2)\dotsm(a+n-1) \textup{ for } n>0.
\end{equation*}
Then for complex numbers $a_i$, $b_j$ and $z$, with none of the $b_j$ a negative integer or zero, 
\begin{equation*}
{_rF_s} \biggl[ \begin{array}{ccccc} a_1, & a_2, & a_3, & \dotsc, & a_r \\
\phantom{a_1} & b_1, & b_2, & \dotsc, & b_s \end{array}
\Big| \; z \biggr]
:=\sum^{\infty}_{n=0}
\frac{\ph{a_1}{n} \ph{a_2}{n} \ph{a_3}{n} \dotsm \ph{a_r}{n}}
{\ph{b_1}{n} \ph{b_2}{n} \dotsm \ph{b_s}{n}}
\; \frac{z^n}{{n!}}.
\end{equation*}
This series satisfies many powerful transformation and summation formulas. The first among them was given by Gauss in 1812. 
\begin{theorem}[Gauss \cite{Ga}]\label{thm_Gauss}
If $Re(c-a-b)>0$, then
\begin{equation*}
{_{2}F_1} \biggl[ \begin{array}{cc} a, & b \vspace{.02in}\\
\phantom{a} & c \end{array}
\Big| \; 1 \biggr]
= \frac{\gf{c} \, \gf{c-a-b}}{\gf{c-a} \,\gf{c-b}}.
\end{equation*}
\end{theorem}

In \cite{W}, Whipple studied series where $r=s+1$, $z=\pm 1$, and $a_1 +1 = a_2 + b_1 = a_3 +b_2 = \dotsm = a_r + b_s$ which he named {\it well-poised}.
 Summation formulas for series of this type already existed before Whipple's work in the case of ${_2F_1}[\cdots | -1]$, due to Kummer \cite{Ku}, and ${_3F_2}[\cdots | 1]$, due to Dixon \cite{Dix}. 
The main results in \cite{W} for transformations of well-poised series in their most general form are as follows.
\begin{theorem}[Whipple \cite{W}]\label{thm_Whipple_4F3}
\begin{multline*}
{_{4}F_3} \biggl[ \begin{array}{cccc} a, & b, &c, &d \vspace{.02in}\\
\phantom{a} & 1+a-b, & 1+a-c, & 1+a-d \end{array}
\Big| \; -1 \biggr]\\[9pt]
= \frac{\gf{1+a-c} \, \gf{1+a-d}}{\gf{1+a} \,\gf{1+a-c-d}} \;
{_{3}F_2} \biggl[ \begin{array}{ccc} 1+\frac{1}{2}a-b, &c, &d \vspace{.02in}\\
\phantom{a} & 1+\frac{1}{2}a, & 1+a-b \end{array}
\Big| \; 1 \biggr].
\end{multline*}
\end{theorem}
\begin{theorem}[Whipple \cite{W}]\label{thm_Whipple_5F4}
If one of $1+\frac{1}{2}a-b$, $c$, $d$, $e$ is a negative integer, then
\begin{multline*}
{_{5}F_4} \biggl[ \begin{array}{ccccc} a, & b, &c, &d, &e \vspace{.02in}\\
\phantom{a} & 1+a-b, & 1+a-c, & 1+a-d, & 1+a-e \end{array}
\Big| \; 1 \biggr]\\[9pt]
= \frac{\gf{1+a-c} \, \gf{1+a-d} \, \gf{1+a-e} \, \gf{1+a-c-d-e}}{\gf{1+a} \,\gf{1+a-d-e}\,\gf{1+a-c-d}\,\gf{1+a-c-e}} \\[6pt]
\cdot {_{4}F_3} \biggl[ \begin{array}{cccc} 1+\frac{1}{2}a-b, &c, &d, &e \vspace{.02in}\\
\phantom{a} & 1+\frac{1}{2}a, & c+d+e-a, &1+a-b \end{array}
\Big| \; 1 \biggr].
\end{multline*}
\end{theorem}
\noindent We note that \cite{W} also includes transformations for well-poised ${_6F_5}[\cdots | -1]$ and ${_7F_6}[\cdots | 1]$ where the `$b$' parameter is specialized to equal $1+\frac{1}{2}a$.  

We now define a finite field analogue of the classical series. Let $\mathbb{F}_{q}$ denote the finite field with $q$, a prime power, elements. Let $\widehat{\mathbb{F}^{*}_{q}}$ denote the group of multiplicative characters of $\mathbb{F}^{*}_{q}$. 
We extend the domain of $\chi \in \widehat{\mathbb{F}^{*}_{q}}$ to $\mathbb{F}_{q}$, by defining $\chi(0):=0$ (including the trivial character $\varepsilon$) and denote $\bar{\chi}$ as the inverse of $\chi$. 
Let $\theta$ be a fixed non-trivial additive character of $\mathbb{F}_q$ and for $\chi \in \widehat{\mathbb{F}^{*}_{q}}$ we define the Gauss sum $g(\chi):= \sum_{x \in \mathbb{F}_q} \chi(x) \theta(x)$.

\begin{defi}\label{def_HypFnFF}
For $A_0,A_1,\dotsc, A_n, B_1 \dotsc, B_n \in \widehat{\mathbb{F}_q^{*}}$ and $x \in \mathbb{F}_q$ define
\begin{equation*}
{_{n+1}F_{n}} {\biggl( \begin{array}{cccc} A_0, & A_1, & \dotsc, & A_n \\
 \phantom{B_0,} & B_1, & \dotsc, & B_n \end{array}
\Big| \; x \biggr)}_{q}^{\star}
:= \frac{1}{q-1}  \sum_{\chi \in \widehat{\mathbb{F}_q^{*}}} 
\prod_{i=0}^{n} \frac{g(A_i \chi)}{g(A_i)}
\prod_{j=1}^{n} \frac{g(\bar{B_j \chi})}{g(\bar{B_j})}
 g(\bar{\chi})
 \chi(-1)^{n+1}
 \chi(x).
 \end{equation*}
\end{defi}

\noindent Throughout the paper we will refer to this function as ${_{n+1}F_{n}}(\cdots)^{\star}$. Using properties of Gauss and Jacobi sums it is easy to see that ${_{n+1}F_{n}}(\cdots)^{\star}$ is independent of the choice of additive character.
We will call the function well-poised when $x = \pm1$ and each $B_j = A_0 \bar{A_j}$, mirroring the conditions in the classical case.

We now state our main results. The first two results are analogues of Whipple's Theorems \ref{thm_Whipple_4F3} and \ref{thm_Whipple_5F4} above. For brevity, if $A \in \widehat{\mathbb{F}_q^{*}}$ is a square we will write $A = \square$.
\begin{theorem}\label{cor_5}
For $A$, $B$, $C$, $D \in \widehat{\mathbb{F}_q^{*}}$,
\begin{multline*}
{_{4}F_3} \biggl( \begin{array}{cccc} A, & B, & C, & D \vspace{.02in}\\
\phantom{A} & A\bar{B}, & A\bar{C}, & A\bar{D}  \end{array}
\Big| \; {-1} \biggr)_{q}^{\star}\\
=
\begin{cases}
0 & \textup{if $A \neq \square$,}\\[4pt]
\dfrac{g(\bar{A}) \, g(\bar{A}CD) }{ g(\bar{A}C) \, g( \bar{A}D)}
\displaystyle\sum_{R^2=A} 
{_{3}F_2} \biggl( \begin{array}{ccc} R\bar{B}, & C, & D \vspace{.02in}\\
\phantom{A} & R, & A\bar{B} \end{array}
\Big| \; 1 \biggr)_{q}^{\star}
& \genfrac{}{}{0pt}{0} {\textup{if $A = \square$,  $A \neq \varepsilon$, $B \neq \varepsilon$,}}{\;\;\; \textup{$B^2 \neq A$ and $CD \neq A$.}}
\end{cases}
\end{multline*}
\end{theorem}

\begin{theorem}\label{cor_6}
For $A$, $B$, $C$, $D$, $E \in \widehat{\mathbb{F}_q^{*}}$, such that, when $A$ is a square, $A \neq \varepsilon$, $B \neq \varepsilon$, $B^2 \neq A$, $CD \neq A$, $CE \neq A$, $DE \neq A$ and $CDE \neq A$,
 \begin{multline*}
{_{5}F_4} \biggl( \begin{array}{ccccc} A, & B, & C, & D, & E \vspace{.02in}\\
\phantom{A} & A\bar{B}, & A\bar{C}, & A\bar{D}, & A\bar{E}  \end{array}
\Big| \; 1 \biggr)_{q}^{\star}\\
=
\begin{cases}
0 & \textup{if $A \neq \square$,}\\[18pt] 
 \displaystyle\frac{g(\bar{A}) g(\bar{A}DE) g(\bar{A}CD) g(\bar{A}CE)}{g(\bar{A}C) g(\bar{A}D) g(\bar{A}E) g(\bar{A}CDE)} \displaystyle\sum_{R^2=A} 
{_{4}F_3} \biggl( \begin{array}{cccc} R\bar{B}, & C, & D,  & E  \vspace{.02in}\\
 \phantom{A,} & R, & \bar{A}CDE, & A\bar{B} \end{array}
 \Big| \; 1 \biggr)_{q}^{\star}\\[24pt]
\qquad + \displaystyle\frac{g(\bar{A}DE) g(\bar{A}CD) g(\bar{A}CE) q}{g(C) g(D) g(E)  g(\bar{A}C) g(\bar{A}D) g(\bar{A}E)} \;
{_{2}F_1} \biggl(  \begin{array}{ccccc} A, & B \vspace{.02in}\\
 \phantom{A,} & A\bar{B} \end{array}
 \Big| \; {-1} \biggr)_{q}^{\star}
& \textup{otherwise.}
 \end{cases}
 \end{multline*}
\end{theorem}

\vspace{0.1in}

\noindent As we will see in Section \ref{sec_Proofs}, the results above are based on the following fundamental result which we use to relate well-poised functions of different orders. 
For $\chi \in \widehat{\mathbb{F}_q^{*}}$ define $\delta(\chi)$ to equal $1$ if $\chi$ is trivial and zero otherwise.
\begin{theorem}\label{thm_7}
For $2 \leq n \in \mathbb{Z}$ and $A_0,A_1,\dotsc, A_n \in \widehat{\mathbb{F}_q^{*}}$,
\begin{multline*}
{_{n+1}F_{n}} {\biggl( \begin{array}{cccc} A_0, & A_1, & \dotsc, & A_n \\
 \phantom{A_0,} & A_0\bar{A_1}, & \dotsc, & A_0\bar{A_n} \end{array}
\Big| \; x \biggr)}_{q}^{\star}
=
\frac{g(\bar{A_0}A_{n-1} A_n)}{g(A_{n-1}) g(A_n) g(\bar{A_0}A_{n-1}) g(\bar{A_0}A_{n})} \times\\
\frac{1}{q-1} \sum_{\psi \in \widehat{\mathbb{F}_q^{*}}} 
g(A_{n-1} \psi) g(A_n \psi) g(\bar{\psi}) g(\bar{A_0}\bar{\psi}) \;
{_{n}F_{n-1}} {\biggl( \begin{array}{ccccc} A_0, & A_1, & \dotsc, & A_{n-2}, & \bar{\psi} \\
 \phantom{A_0,} & A_0\bar{A_1}, & \dotsc, & A_0\bar{A_{n-2}}, & A_0 \psi \end{array}
\Big| \; {-x} \biggr)}_{q}^{\star}\\
+
\frac{q(q-1) A_nA_{n-1}(-1) \delta({\bar{A_0}A_{n-1} A_n})}{g(A_{n-1}) g(A_n) g(\bar{A_0}A_{n-1}) g(\bar{A_0}A_{n})} \;
{_{n-1}F_{n-2}} {\biggl( \begin{array}{cccc} A_0, & A_1, & \dotsc, & A_{n-2} \\
 \phantom{A_0,} & A_0\bar{A_1}, & \dotsc, & A_0\bar{A_{n-2}} \end{array}
\Big| \; x \biggr)}_{q}^{\star}
\end{multline*}
\end{theorem}


\noindent Based on Theorem \ref{thm_7} we will show by induction that all well-poised functions of the form ${_{n+1}F_{n}}(\cdots | ({-1})^n)^{\star}$ equal zero if the leading top line parameter is not a square.
\begin{cor}\label{cor_7}
 For $0 \leq n \in \mathbb{Z}$ and characters $A_0,A_1,\dotsc, A_n \in \widehat{\mathbb{F}_q^{*}}$ such that $A_0$ is not a square,
\begin{equation*}
{_{n+1}F_{n}} {\biggl( \begin{array}{cccc} A_0, & A_1, & \dotsc, & A_n \\
 \phantom{A_0,} & A_0\bar{A_1}, & \dotsc, & A_0\bar{A_n} \end{array}
\Big| \; (-1)^n \biggr)}_{q}^{\star}
=0.
\end{equation*}
\end{cor}

We also have other results which are analogues of various classical summation formulas.
The first of these is an analogue of Gauss' Theorem \ref{thm_Gauss} above and can easily be derived from a character sum evaluation of Helversen-Pasotto (see Corollary \ref{cor_HP} in Section 2). 
\begin{theorem}[Helversen-Pasotto \cite{Hp1}]\label{thm_1a}
For $A$, $B$, $C \in \widehat{\mathbb{F}_q^{*}}$ such that $AB \neq C$,
\begin{equation*}
{_{2}F_1} \biggl( \begin{array}{cc} A, & B \vspace{.02in}\\
\phantom{A} & C \end{array}
\Big| \; 1 \biggr)_{q}^{\star}
= \frac{g(A\bar{C}) \, g(B\bar{C})}{ g(\bar{C}) \, g(AB \bar{C})}.
\end{equation*}
\end{theorem}

\noindent The next two results are analogues of Kummer's theorem (see \cite{Ba}  2.3 (1)) and  Dixon's theorem (see \cite{Ba} 3.1 (1)) respectively.  
We note that they can be derived from (4.11) and Theorem 4.37 of Greene \cite{G}, via Proposition \ref {prop_FtoGreene} for most choices of parameters and on a case by case basis otherwise. In turn, the latter result of Greene is closely related to a character sum evaluation of Evans \cite{E2}. However as we will see in Section \ref{sec_Proofs} our method of proof is different. 
\begin{theorem}[cf. Greene \cite{G}]\label{cor_3}
For $A$, $B \in \widehat{\mathbb{F}_q^{*}}$ such that $A \neq \varepsilon$,
\begin{equation*}
{_{2}F_1} \biggl( \begin{array}{cc} A, & B \vspace{.02in}\\
\phantom{A} & A\bar{B} \end{array}
\Big| \; {-1} \biggr)_{q}^{\star}
=
\begin{cases}
0 & \textup{if $A \neq \square$,}\\[4pt]
\displaystyle\sum_{R^2=A} \dfrac{ g(\bar{A}) \, g(\bar{R}B)}{g(\bar{R}) \, g(\bar{A}B)} 
& \textup{otherwise.}
\end{cases}
\end{equation*}
\end{theorem}

\begin{theorem}[cf. Greene \cite{G}, cf. Evans \cite{E2}]\label{cor_4}
For $A$, $B$, $C \in \widehat{\mathbb{F}_q^{*}}$ such that $A \neq \varepsilon$ and $(BC)^2 \neq A$,
\begin{equation*}
{_{3}F_2} \biggl( \begin{array}{ccc} A, & B, & C \vspace{.02in}\\
\phantom{A} & A\bar{B}, & A\bar{C} \end{array}
\Big| \; 1 \biggr)_{q}^{\star}
=
\begin{cases}
0 & \textup{if $A \neq \square$,}\\[4pt]
\displaystyle\sum_{R^2=A} \dfrac{ g(\bar{A}) \, g(\bar{R}B) \, g(\bar{R}C) \, g(\bar{A}BC)}{g(\bar{R}) \, g(\bar{A}B) \, g(\bar{A}C) \, g(\bar{R}BC)} 
& \textup{otherwise.}
\end{cases}
\end{equation*}
\end{theorem}

As mentioned earlier, there are other finite field analogues of the classical series, most notably those defined by Greene \cite{G} and Katz \cite{Ka}.  
The function ${_{n+1}F_{n}}(\cdots)^{\star}$ is a normalization of Katz's function and is also closely related to a normalization of Greene's function (though significantly different for certain choices of parameters). 
We will discuss these relationships in more detail in Section 2 and also the motivation for the definition of ${_{n+1}F_{n}}(\cdots)^{\star}$.

However we note at this stage that Greene's function has featured in results in many areas and that all these results can be reformulated in terms of ${_{n+1}F_{n}}(\cdots)^{\star}$: character sum evaluations \cite{GS, EL}; 
finite field versions of the Lagrange inversion formula \cite{G2}; the representation theory of SL($2, \mathbb{F}_q$) \cite{G3}; 
evaluating the number of points over $\mathbb{F}_{p}$ of certain algebraic varieties \cite{AO, F2, DMC2}; 
proofs of supercongruences \cite{A, AO, K, MO, M1, M2, M}; traces of Hecke operators \cite{FOP, F2}; formulas for Ramanujan's $\tau$-function \cite{F2, P}; and, relationships with Fourier coefficients of certain other modular forms \cite{A, AO, E, DMC2, M}. 

In particular, all of the results cited above for relationships with Fourier coefficients of modular forms can be restated in terms of ${_{n+1}F_{n}}(\cdots)^{\star}$. For example, let $\gamma(n)$ be given by
$
\eta^{4}(2z) \eta^{4}(4z)=\sum_{n=1}^{\infty} \gamma(n)q^n \in S_4(\Gamma_0(8))
$
where $q:=e^{2 \pi i z}$ and
$\eta(z):=q^{\frac{1}{24}} \prod_{n=1}^{\infty}(1-q^n)
$ is Dedekind's eta function. Let $\phi \in \widehat{\mathbb{F}_p^{*}}$ be the character of order $2$. Then one of the main results in \cite{AO} can be re-written in terms of a well-poised ${_{4}F_{3}}(\cdots)^{\star}$.  
\begin{theorem}[Ahlgren and Ono \cite{AO}] \label{thm_AO}
If $p$ is an odd prime, then
\begin{equation*}
{_{4}F_{3}} {\biggl( \begin{array}{cccc} \phi, & \phi, & \phi, & \phi \\
\phantom{\phi} & \varepsilon, & \varepsilon, & \varepsilon \end{array}
\Big| \; 1 \biggr)}_{p}^{\star}
=\gamma(p) + p.
\end{equation*}
\end{theorem}

\noindent The corresponding version of Theorem \ref{thm_AO} using Greene's function features an additional factor of $-p^3$ on the left hand side. 
Factors of this type are common in results involving Greene's function and one advantage of using ${_{n+1}F_{n}}(\cdots)^{\star}$ is that these factors are not required, leading to cleaner results. 
Other advantages of using this new definition is that it leads to less restrictions on the parameters in our transformations; and the parameters in any given function can be permuted without changing the value of the function, which is a key feature of the classical series but not of Greene's function.

The rest of this paper is organized as follows. In Section \ref{sec_Rel} we will discuss the motivation for Definition \ref{def_HypFnFF} and its relationships with other functions in this area.
Section \ref{sec_Prelim} introduces some preliminary details on Gauss and Jacobi sums which we will use in proving our results in Section \ref{sec_Proofs}.
We then make some concluding remarks in Section \ref{sec_cr}.


\section{Relationships with other Hypergeometric Functions over Finite Fields.}\label{sec_Rel}
In this section we outline the relationship between the hypergeometric function over finite fields defined in Definition \ref{def_HypFnFF} and those defined by Greene \cite{G} and Katz \cite{Ka}.

We start by defining the function of Greene. For $A$, $B  \in \widehat{\mathbb{F}^{*}_{q}}$, define
$$\bin{A}{B}:= \frac{B(-1)}{q} \sum_{x \in \mathbb{F}_{q}} A(x) \bar{B}(1-x).$$ 
Then for $A_0,A_1,\dotsc, A_n$, $B_1, \dotsc, B_n  \in \widehat{\mathbb{F}^{*}_{q}}$ and $x \in \mathbb{F}_{q}$, define
\begin{equation*}
{_{n+1}F_n} {\left( \begin{array}{cccc} A_0, & A_1, & \dotsc, & A_n \\
\phantom{A_0} & B_1, & \dotsc, & B_n \end{array}
\Big| \; x \right)}^{G}_{q}
:= \frac{q}{q-1} \sum_{\chi  \in \widehat{\mathbb{F}^{*}_{q}}} \binom{A_0 \chi}{\chi} \prod_{i=1}^{n} \binom{A_i \chi}{B_i \chi} \chi(x) .
\end{equation*}

\noindent Greene found many transformation and summation formulas analogous to those in the classical case. For example, the following is an analogue of Gauss' Theorem \ref{thm_Gauss}. 
We quote this result with an appropriate normalization factor which Greene noted would be required to state the result in a comparable form to the classical case.  
\begin{theorem}[Greene \cite{G}]\label{thm_G_Gauss}
If $A \neq \varepsilon$, $B \neq C$ and $AB \neq C$ then
\begin{equation*}
\binom{B}{C}^{-1}
{_{2}F_1} \biggl( \begin{array}{cc} A, & B \vspace{.02in}\\
\phantom{A} & C \end{array}
\Big| \; 1 \biggr)^G_q
= \frac{g(A\bar{C}) \, g(B\bar{C})}{ g(\bar{C}) \, g(AB \bar{C})}.
\end{equation*}
\end{theorem}


\noindent We now recall a character sum evaluation of Helversen-Pasotto. 
\begin{theorem}[Helversen-Pasotto \cite{Hp1}]\label{thm_HP}
For $A$, $B$, $C$, $D \in \widehat{\mathbb{F}_q^{*}}$,
\begin{multline*} 
\frac{1}{q-1} \sum_{\chi \in \widehat{\mathbb{F}_q^{*}}} g(A \chi) g(B \chi) g(C \bar{\chi}) g(D \bar{\chi})\\
 = \frac{g(AC) g(AD) g(BC) g(BD)}{g(ABCD)} +q(q-1) AB(-1) \delta(ABCD).
\end{multline*}
\end{theorem}

\noindent Choosing $D$ to be the trivial character, replacing $C$ by $\bar{C}$ and introducing an appropriate factor we get the following corollary.
\begin{cor}\label{cor_HP}
For $A$, $B$, $C \in \widehat{\mathbb{F}_q^{*}}$ such that $AB \neq C$,
\begin{equation*}
\frac{1}{q-1} \sum_{\chi \in \widehat{\mathbb{F}_q^{*}}} \frac{g(A \chi)}{g(A)} \frac{g(B \chi)}{g(B)} \frac{g(\bar{C \chi})}{g(\bar{C})} g( \bar{\chi})
= \frac{g(A\bar{C}) \, g(B\bar{C})}{ g(\bar{C}) \, g(AB \bar{C})}.
\end{equation*}
\end{cor}

\noindent This can be also be viewed as an analogue of Gauss' Theorem \ref{thm_Gauss} but has fewer restrictions on the parameters than Theorem \ref{thm_G_Gauss}. 
We have therefore framed Definition \ref{def_HypFnFF} as a generalization of the left hand side of Corollary \ref{cor_HP} and, in general, transformations involving this function require fewer restrictions on the parameters than corresponding results expressed using Greene's function.
For example, we can restate Theorem 4.37 of \cite{G} in a comparable form to Theorem \ref{cor_4} but we require more conditions on the parameters.
\begin{theorem}[Greene, \cite {G} Thm 4.37]
For $A$, $B$, $C \in \widehat{\mathbb{F}_q^{*}}$ such that $A \neq \varepsilon$, $B \neq \varepsilon$, $C \neq \varepsilon$, $BC \neq A$ and $(BC)^2 \neq A$,
\begin{multline*}
\binom{B}{A\bar{C}}^{-1} \binom{A}{A\bar{B}}^{-1} {_{3}F_2} \biggl( \begin{array}{ccc} C, & B, & A \vspace{.02in}\\
\phantom{A} & A\bar{C}, & A\bar{B} \end{array}
\Big| \; 1 \biggr)^{G}_{q}\\[8pt]
=
\begin{cases}
0 & \textup{if $A \neq \square$,}\\[4pt]
\displaystyle\sum_{R^2=A} \dfrac{ g(\bar{A}) \, g(\bar{R}B) \, g(\bar{R}C) \, g(\bar{A}BC)}{g(\bar{R}) \, g(\bar{A}B) \, g(\bar{A}C) \, g(\bar{R}BC)} 
& \textup{otherwise.}
\end{cases}
\end{multline*}
\end{theorem}

\noindent Many of the transformations we develop in this paper are based on summation properties of products of  Gauss sums. 
Greene's function is essentially defined using Jacobi sums and often it is necessary to impose conditions on the parameters to relate the Jacobi sums to the required product of Gauss sums. 
Defining ${_{n+1}F_{n}}(\cdots)^{\star}$ purely in terms of Gauss sums strips out the need for these conditions.
The following proposition relates the two functions when certain conditions on the parameters are satisfied.
\begin{prop}\label{prop_FtoGreene}
If $A_0 \neq \varepsilon$ and $A_i \neq B_i$ for each $1 \leq i \leq n$ then 
\begin{equation*}
{_{n+1}F_{n}} {\biggl( \begin{array}{cccc} A_0, & A_1, & \dotsc, & A_n \\
 \phantom{B_0,} & B_1, & \dotsc, & B_n \end{array}
\Big| \; x \biggr)}_{q}^{\star}
=
\left[\prod_{i=1}^{n} \binom{A_i}{B_i}^{-1}\right]
{_{n+1}F_{n}} {\biggl( \begin{array}{cccc} A_0, & A_1, & \dotsc, & A_n \\
 \phantom{B_0,} & B_1, & \dotsc, & B_n \end{array}
\Big| \; x \biggr)}^{G}_{q}.
\end{equation*}
\end{prop}
\noindent When these conditions are not satisfied, the relationship is not quite as straightforward. For example if $A_0 \neq \varepsilon$ and $A_i \neq B_i$ for each $1 \leq i \leq n-1$ but $A_n=B_n\neq \varepsilon$ , then
\begin{multline*}
{_{n+1}F_{n}} {\biggl( \begin{array}{ccccc} A_0, & A_1, & \dotsc, & A_{n-1}, & A_n \\
 \phantom{A_0,} & B_1, & \dotsc, & B_{n-1} , & A_n \end{array}
\Big| \; x \biggr)}_{q}^{\star}\\
=
\left[\prod_{i=1}^{n} \binom{A_i}{B_i}^{-1}\right]
{_{n+1}F_{n}} {\biggl( \begin{array}{ccccc} A_0, & A_1, & \dotsc, & A_{n-1}, & A_n \\
 \phantom{A_0,} & B_1, & \dotsc, & B_{n-1} , & A_n \end{array}
\Big| \; x \biggr)}^{G}_{q}\\
+(q-1) \binom{A_0 \bar{A_n}}{\bar{A_n}} \left[\prod_{i=1}^{n-1}\binom{A_i \bar{A_n}}{B_i \bar{A_n}} \binom{A_i}{B_i}^{-1}\right] \bar{A_n}(x).
\end{multline*}

We now consider the `hypergeometric sum' defined by Katz (see \cite{Ka} Ch 8.2).
For $t \in \mathbb{F}_q^{*}$ and, for $m,n \in \mathbb{Z}^+$, let 
\begin{equation*}
V(t,n,m) = \left\{x_i, y_j \in \mathbb{F}_q^{*}: x_1 x_2 \dotsc x_n = t y_1 y_2 \dotsc y_m\right\}.
\end{equation*}
If $\theta$ is a fixed non-trivial additive character of $\mathbb{F}_q$ and $A_1,A_2,\dotsc, A_n, B_1, B_2 \dotsc, B_m \in \widehat{\mathbb{F}_q^{*}}$ then the `hypergeometric sum' is defined  by
\begin{equation*}
{_{n}F_m} {\Biggl( \begin{array}{cccc} A_1, & A_2, & \dotsc, & A_n \\
B_1 & B_2, & \dotsc, & B_m \end{array}
\bigg| \; t \Biggr)}_{q}^{K}
:= \sum_{V} \theta\Biggl(\sum_{i=1}^n x_i - \sum_{j=1}^m y_j\Biggr) \prod_{i=1}^{n} A_i(x_i)  \prod_{j=1}^{m} \bar{B}_j(y_j).
\end{equation*}
This can be transformed by multiplicative Fourier inversion to get
\begin{equation*}\label{HS_Katz}
{_{n}F_m} {\Biggl( \begin{array}{cccc} A_1, & A_2, & \dotsc, & A_n \\
B_1 & B_2, & \dotsc, & B_m \end{array}
\bigg| \; t \Biggr)}_{q}^{K}
=  \frac{1}{q-1} \sum_{\chi}  \bar{\chi}(t) \prod_{i=1}^{n} g(A_i \chi) \prod_{j=1}^{m} g(\bar{B_j \chi}) \; {B_j \chi}(-1).
\end{equation*}

\noindent Thus we have the following direct relation between Katz's function and ${_{n+1}F_{n}}(\cdots)^{\star}$.
\begin{prop}\label{prop_relKatz}
For $A_0,A_1,\dotsc, A_n$ and $B_1 \dotsc, B_n \in \widehat{\mathbb{F}_q^{*}}$,
\begin{multline*}
{_{n+1}F_{n}} {\biggl( \begin{array}{cccc} A_0, & A_1, & \dotsc, & A_n \\
 \phantom{B_0,} & B_1, & \dotsc, & B_n \end{array}
\Big| \; x \biggr)}_{q}^{\star}\\
=
\left[\frac{1}{g(A_0)}\prod_{i=1}^{n} \frac{B_i(-1)}{g(A_i) \, g(\bar{B_i})}\right]
{_{n+1}F_{n+1}} {\biggl( \begin{array}{cccc} A_0, & A_1, & \dotsc, & A_n \\
 \varepsilon & B_1, & \dotsc, & B_n \end{array}
\Big| \; \frac{1}{x} \biggr)}^{K}_{q}.
\end{multline*}
\end{prop}

\noindent Because the relationship in Proposition \ref{prop_relKatz} is unconditional, all results from Section 1 can be rewritten in terms of Katz's function. 
For example Theorem \ref{thm_7} can be restated as
\begin{multline*}
{_{n+1}F_{n+1}} {\biggl( \begin{array}{cccc} A_0, & A_1, & \dotsc, & A_n \\
 \varepsilon, & A_0\bar{A_1}, & \dotsc, & A_0\bar{A_n} \end{array}
\Big| \; x \biggr)}^{K}_{q}
=
g(\bar{A_0}A_{n-1} A_n) \bar{A_0}A_{n-1} A_n(-1) \times\\
\frac{1}{q-1} \sum_{\psi \in \widehat{\mathbb{F}_q^{*}}} 
g(A_{n-1} \psi) g(A_n \psi) \psi(-1) \;
{_{n}F_{n}} {\biggl( \begin{array}{ccccc} A_0, & A_1, & \dotsc, & A_{n-2}, & \bar{\psi} \\
 \varepsilon, & A_0\bar{A_1}, & \dotsc, & A_0\bar{A_{n-2}}, & A_0 \psi \end{array}
\Big| \; -x \biggr)}^{K}_{q}\\
+
q(q-1)  \delta({\bar{A_0}A_{n-1} A_n}) \;
{_{n-1}F_{n-1}} {\biggl( \begin{array}{cccc} A_0, & A_1, & \dotsc, & A_{n-2} \\
 \varepsilon, & A_0\bar{A_1}, & \dotsc, & A_0\bar{A_{n-2}} \end{array}
\Big| \; x \biggr)}^{K}_{q}.
\end{multline*}
We note that this is neater than the formula in Theorem \ref{thm_7}. However, we choose to use ${_{n+1}F_{n}}(\cdots)^{\star}$ as the resulting transformations more closely mirror their classical analogues.
Also it leads to neater relationships in those results related to Fourier coefficients of modular forms. 


\section{Preliminaries}\label{sec_Prelim}
In this section we recall some properties of Gauss and Jacobi sums. For further details see \cite{BEW}, noting that we have adjusted results to take into account $\varepsilon(0)=0$. 
As noted in Section \ref{sec_Intro} we  let $\widehat{\mathbb{F}^{*}_{q}}$ denote the group of multiplicative characters of $\mathbb{F}^{*}_{q}$. 
We extend the domain of $\chi \in \widehat{\mathbb{F}^{*}_{q}}$ to $\mathbb{F}_{q}$, by defining $\chi(0):=0$ (including the trivial character $\varepsilon$) and denote $\bar{\chi}$ as the inverse of $\chi$. 
We then have the following orthogonal relations. For a character $\chi$ of $\mathbb{F}_{q}$,
\begin{equation}\label{for_TOrthEl}
\sum_{x \in \mathbb{F}_q} \chi(x)=
\begin{cases}
q-1 & \text{if $\chi = \varepsilon$}  ,\\
0 & \text{if $\chi \neq \varepsilon$}  ,
\end{cases}
\end{equation}
and, for $x \in \mathbb{F}_q$, we have
\begin{equation}\label{for_TOrthCh}
\sum_{\chi \in \widehat{\mathbb{F}_q^{*}}} \chi(x)=
\begin{cases}
q-1 & \text{if $x=1$}  ,\\
0 & \text{if $x \neq 1$}  .
\end{cases}
\end{equation}

\noindent Let $\theta$ be a fixed non-trivial additive character of $\mathbb{F}_q$ and recall that for $\chi \in \widehat{\mathbb{F}^{*}_{q}}$ we define the Gauss sum  by $g(\chi):= \sum_{x \in \mathbb{F}_q} \chi(x) \theta(x)$.
The following important result gives a simple expression for the product of two Gauss sums based on inverse characters.
For a character $\chi$ of $\mathbb{F}_q$ we have
\begin{equation}\label{for_GaussConj}
g(\chi)g(\bar{\chi})=
\begin{cases}
\chi({-1}) q & \text{if } \chi \neq \varepsilon,\\
1 & \text{if } \chi= \varepsilon.
\end{cases}
\end{equation}
\noindent Recall also that for $\chi, \psi \in \widehat{\mathbb{F}^{*}_{q}}$ we define the Jacobi sum by $J(\chi,\psi):=\sum_{t \in \mathbb{F}_q} \chi(t) \psi(1-t).$
\noindent We can relate Jacobi sums to Gauss sums.
For $\chi$, $\psi \in \widehat{\mathbb{F}_q^{*}}$ not both trivial,
\begin{equation}\label{for_JactoGauss}
J(\chi, \psi)=
\begin{cases}
\dfrac{g(\chi)g(\psi)}{g(\chi \psi)}
& \qquad \text{if } \chi \psi \neq \varepsilon,\\[18pt]
-\dfrac{g(\chi)g(\psi)}{q}
&\qquad \text{if }\chi \psi = \varepsilon \: .
\end{cases}
\end{equation}

We now develop a couple of preliminary results which will be used in Section \ref{sec_Proofs}.
\begin{prop}\label{prop_sumJac}
For characters $A$ and $B \in \widehat{\mathbb{F}_q^{*}}$,
\begin{equation*}
 \sum_{\chi \in \widehat{\mathbb{F}_q^{*}}} J(A\chi, B \bar{\chi}) \chi({-1}) =0.
\end{equation*}
\end{prop}

\begin{proof}
By definition of Jacobi sums and using (\ref{for_TOrthCh}) we see that
\begin{align*}
\sum_{\chi \in \widehat{\mathbb{F}_q^{*}}} J(A\chi, B \bar{\chi}) \chi(-1)
&= \sum_{\chi \in \widehat{\mathbb{F}_q^{*}}} \sum_{t \in \mathbb{F}_{q}}
A \chi(t) B \bar{\chi}(1-t) \chi(-1)\\
&= \sum_{t \in \mathbb{F}_{q}} A(t) B(1-t)
\sum_{\chi \in \widehat{\mathbb{F}_q^{*}}} 
\chi \left(\frac{-t}{1-t}\right)\\
&=0.
\end{align*}
\end{proof}

\begin{prop}\label{thm_2}
For characters $A$ and $B \in \widehat{\mathbb{F}_q^{*}}$,
\begin{equation*}
\frac{1}{q-1}  \sum_{\chi \in \widehat{\mathbb{F}_q^{*}}} g(A \chi) g(B \bar{\chi}) \chi(-1)
=
\begin{cases}
0 & \textup{if $AB \neq \varepsilon$,}\\
(q-1)A({-1}) & \textup{if $AB=\varepsilon$.}
\end{cases}
\end{equation*}
\end{prop}

\begin{proof}
Applying (\ref{for_JactoGauss}) we see that for $AB \neq \varepsilon$
\begin{equation*}
\frac{1}{q-1}  \sum_{\chi \in \widehat{\mathbb{F}_q^{*}}} g(A \chi) g(B \bar{\chi}) \chi({-1})
= \frac{g(AB)}{q-1}  \sum_{\chi \in \widehat{\mathbb{F}_q^{*}}} J(A\chi, B \bar{\chi}) \chi({-1}),
\end{equation*}
which equals $0$ by Proposition \ref{prop_sumJac}. If $B=\bar{A}$, then by (\ref{for_GaussConj})
\begin{align*}
\frac{1}{q-1}  \sum_{\chi \in \widehat{\mathbb{F}_q^{*}}} g(A \chi) g(\bar{A\chi}) \chi({-1})
&= \frac{1}{q-1} \sum_{\substack{\chi \in \widehat{\mathbb{F}_q^{*}}\\ \chi \neq \bar{A}}} A({-1}) q +\frac{ \bar{A}({-1})}{q-1}\\
&=\frac{A({-1})}{q-1} [q(q-2)+1] = A({-1}) (q-1).
\end{align*}
\end{proof}


\section{Proofs}\label{sec_Proofs}
In this section we prove our results, starting with Theorem \ref{cor_3}. We will then prove Theorem \ref{thm_7} which will be the starting point for proving the other results.
As we will see, Theorem \ref{thm_7} is proved by using the analogue of Gauss' theorem (i.e., Theorem \ref{thm_HP}) in reverse. This strategy mirrors the method used by Whipple in proving his results.

\begin{proof}[Proof of Theorem \ref{cor_3}]
We will prove the slightly more general result which has no restrictions on the parameters.
\begin{theorem}\label{thm_3}
 For $A$, $B \in \widehat{\mathbb{F}_q^{*}}$,
\begin{equation*}
{_{2}F_1} \biggl( \begin{array}{cc} A, & B \vspace{.02in}\\
\phantom{A} & A\bar{B} \end{array}
\Big| \; {-1} \biggr)_{q}^{\star}
=
\begin{cases}
0 & \textup{if $A \neq \square$,}\\[4pt]
\displaystyle\sum_{R^2=A} \dfrac{ g(R) \, g(\bar{R}B) R(-1)}{g(A) \, g(\bar{A}B)} 
& \textup{otherwise.}
\end{cases}
\end{equation*}
\end{theorem}

\noindent By definition
\begin{equation}\label{wp_2F1}
{_{2}F_1} \biggl( \begin{array}{cc} A, & B \vspace{.02in}\\
\phantom{A} & A\bar{B} \end{array}
\Big| \; {-1} \biggr)_{q}^{\star}
= 
\frac{1}{q-1}  \sum_{\chi \in \widehat{\mathbb{F}_q^{*}}} 
\frac{g(A \chi)}{g(A)}
\frac{g(B \chi)}{g(B)}
 \frac{g(\bar{A} B \bar{\chi})}{g(\bar{A}B)}
g(\bar{\chi})
\chi(-1).
\end{equation}

\noindent By Theorem \ref{thm_HP}
\begin{multline}\label{HP_reverse_2F1}
g(A \chi) g(B \chi) g(\bar{A} B \bar{\chi}) g( \bar{\chi})\\
=
\frac{g(B)}{q-1} \sum_{\psi \in \widehat{\mathbb{F}_q^{*}}} 
g(A \psi) g(B \psi) g(\bar{A \chi \psi}) g(\chi \bar{\psi})
-g(B)q(q-1) AB(-1) \delta(B).
\end{multline}

\noindent Substituting (\ref{HP_reverse_2F1}) in to (\ref{wp_2F1}) yields
\begin{multline*}
{_{2}F_1} \biggl( \begin{array}{cc} A, & B \vspace{.02in}\\
\phantom{A} & A\bar{B} \end{array}
\Big| \; {-1} \biggr)_{q}^{\star}
= 
\frac{1}{q-1}  \sum_{\psi \in \widehat{\mathbb{F}_q^{*}}}
 \frac{g(A \psi) g(B \psi)}{g(A)g(\bar{A}B)}
\frac{1}{q-1}  \sum_{\chi \in \widehat{\mathbb{F}_q^{*}}} 
g(\bar{\psi} \chi ) g(\bar{A \psi \chi }) \chi(-1)\\
 -\frac{qAB(-1) \delta(B)}{g(A) g(\bar{A}B)}  \sum_{\chi \in \widehat{\mathbb{F}_q^{*}}} \chi(-1).
\end{multline*}
By (\ref{for_TOrthCh}) the second term above is 0. Using Proposition \ref{thm_2} to evaluate the first term yields Theorem \ref{thm_3}. If $A\neq \varepsilon$ then $R\neq \varepsilon$ and Theorem \ref{cor_3} follows on applying  (\ref {for_GaussConj}).
\end{proof}

\begin{proof}[Proof of Theorem \ref{thm_7}]
By definition
\begin{multline}\label{wp_gen}
{_{n+1}F_{n}} {\biggl( \begin{array}{cccc} A_0, & A_1, & \dotsc, & A_n \\
 \phantom{A_0,} & A_0\bar{A_1}, & \dotsc, & A_0\bar{A_n} \end{array}
\Big| \; x \biggr)}_{q}^{\star}\\
= 
\frac{1}{q-1}  \sum_{\chi \in \widehat{\mathbb{F}_q^{*}}} 
\frac{g(A_0 \chi)}{g(A_0)}
\prod_{i=0}^{n} \frac{g(A_i \chi)}{g(A_i)} \frac{g(\bar{A_0} A_i \bar{\chi})}{g(\bar{A_0}A_i)}
g(\bar{\chi})
\chi(-1)^{n+1}
\chi(x).
\end{multline}

\noindent By Theorem \ref{thm_HP}
\begin{multline}\label{HP_reverse_gen}
g(A_{n-1} \chi) g(A_{n} \chi) g(\bar{A_0} A_{n-1} \bar{\chi}) g(\bar{A_0} A_{n} \bar{\chi})
=\\
\frac{g(\bar{A_0}A_{n-1} A_n)}{q-1} \sum_{\psi \in \widehat{\mathbb{F}_q^{*}}} 
g(A_{n-1} \psi) g(A_n \psi) g(\bar{A_0 \chi \psi}) g(\chi \bar{\psi})\\
-g(\bar{A_0}A_{n-1} A_n)q(q-1) A_nA_{n-1}(-1) \delta({\bar{A_0}A_{n-1} A_n}).
\end{multline}
Substituting (\ref{HP_reverse_gen}) in to (\ref{wp_gen}) and tidying up yields the result. 
\end{proof}

\begin{proof}[Proof of Theorem \ref{cor_4}]
Again, we prove a more general result from which Theorem \ref{cor_4} can easily be derived.
\begin{theorem}\label{thm_4}
 For $A$, $B$, $C \in \widehat{\mathbb{F}_q^{*}}$,
\begin{multline*}
{_{3}F_2} \biggl( \begin{array}{ccc} A, & B, & C \vspace{.02in}\\
\phantom{A} & A\bar{B}, & A\bar{C} \end{array}
\Big| \; 1 \biggr)_{q}^{\star}\\
=
\begin{cases}
0 & \textup{if $A \neq \square$,}\\[9pt]
\displaystyle\sum_{R^2=A} \dfrac{ g(\bar{A}BC)  \, g(\bar{R}B) \, g(\bar{R}C) \, g(R) \, g(R\bar{BC})}{ g(\bar{A}B) \, g(\bar{A}C) \,  g(A) BC(-1) \, q} 
& \genfrac{}{}{0pt}{0}{\textup{if $A = \square$ and $(BC)^2 \neq A$, } \textit{or} \phantom{AAA\;\;}}{\textup{if $(BC)^2 = A \neq BC$, $B \neq \varepsilon$, $C \neq \varepsilon$,}}\\[21pt]
-q+3 &\genfrac{}{}{0pt}{0} {\textup{if $(BC)^2 = A  \neq  BC$, $A \neq \varepsilon$, $B$ or $C= \varepsilon$, } \textit{or}}{\textup{if $A=\varepsilon$, $C=\bar{B}$, $B \neq \varepsilon, \phi$,}\phantom{AAAAAAAAAA}}\\[18pt]
-q^2+2q+1 & \textup{if $A=\varepsilon$, $BC=\phi$, $B=\varepsilon$ or $B=\phi$,}\\[12pt] 
-q^3+q^2+q+1 & \textup{if $A=B=C=\varepsilon$,}\\[12pt]
\frac{-q^2+2q+1}{q} & \textup{if $A=\varepsilon$, $B=C=\phi$.}
\end{cases}
\end{multline*}
\end{theorem}

\noindent Applying Theorem \ref{thm_7} we see that
\begin{multline}\label{for_thm4_thm7}
{_{3}F_2} \biggl( \begin{array}{ccc} A, & B, & C \vspace{.02in}\\
\phantom{A} & A\bar{B}, & A\bar{C} \end{array}
\Big| \; 1 \biggr)_{q}^{\star}\\
=
\frac{1}{q-1} \sum_{\psi \in \widehat{\mathbb{F}_q^{*}}} 
\frac{g(\bar{A}B C) g(B \psi) g(C \psi) g(\bar{\psi}) g(\bar{A}\bar{\psi})}{g(B) g(C) g(\bar{A}B) g(\bar{A}C)} \:
{_{2}F_{1}} {\biggl( \begin{array}{cc} A, & \bar{\psi} \\
 \phantom{A,}  &  A\psi \end{array}
\Big| \; {-1} \biggr)}_{q}^{\star}\\
+
\frac{q \, BC(-1) \delta({\bar{A}BC})}{g(B) g(C) g(\bar{A}B) g(\bar{A}C)}
\sum_{\chi \in \widehat{\mathbb{F}_q^{*}}} 
\frac{g(A \chi)}{g(A)}
g(\bar{\chi})
\chi(-1).
\end{multline}
For brevity we will refer to the two terms on the right-hand side of (\ref{for_thm4_thm7}) as $T_1$ and $T_2$ respectively. Using Proposition \ref{thm_2} we get that
\begin{equation}\label{for_thm7_T2}
T_2 = 
\begin{cases}
0 & \textup{if $A \neq BC$} \textit{ or } \textup{if $A=BC\neq \varepsilon$}  \\[4pt]
-\frac{{(q-1)}^2}{q} & \textup{if $A=BC=\varepsilon$ and $B \neq \varepsilon$}\\[4pt]
-q{(q-1)}^2 & \textup{if $A=B=C=\varepsilon$.} 
\end{cases}
\end{equation}

\noindent We now focus on $T_1$ and use Theorem \ref{thm_3} to evaluate the $_2F_1$. Therefore
\begin{equation*}
 T_1 = 
\begin{cases}
0 & \textup{if $A \neq \square$}\\[4pt]
\displaystyle\sum_{R^2=A} 
\dfrac{g(\bar{A}B C) \, g(R) \,  R(-1)}{g(B) g(C) g(\bar{A}B) g(\bar{A}C) g(A)}
\frac{1}{q-1} \sum_{\psi \in \widehat{\mathbb{F}_q^{*}}} 
g(B \psi) g(C \psi) g(\bar{\psi}) g(\bar{R \psi})
& \textup{otherwise.}
\end{cases}
\end{equation*}
We now assume $A$ is a square and use Theorem \ref{thm_HP} to simplify $T_1$ in this case. This yields
\begin{multline*}
T_1=
\displaystyle\sum_{R^2=A} 
\dfrac{g(\bar{A}B C) \, g(R) \, g(\bar{R}B) \, g(\bar{R}C) \, R(-1)}{g(\bar{A}B) g(\bar{A}C) g(A) g(\bar{R}BC)}
+
\displaystyle\sum_{R^2=A} 
\dfrac{g(\bar{A}B C) \, g(R) \, q(q-1) \delta(\bar{R}BC)}{g(B) g(C) g(\bar{A}B) g(\bar{A}C) g(A)}.
\end{multline*}

\noindent If $(BC)^2 \neq A$ the second sum equals zero and using (\ref{for_GaussConj}) we see that
\begin{equation*}
T_1=
\displaystyle\sum_{R^2=A} \dfrac{ g(\bar{A}BC)  \, g(\bar{R}B) \, g(\bar{R}C) \, g(R) \, g(R\bar{BC})}{ g(\bar{A}B) \, g(\bar{A}C) \,  g(A) BC(-1) \, q} .
\end{equation*}
\noindent If $A$ is a square then there are exactly two characters $R_1$ and $R_2$ such that $(R_i)^2=A$ (and $R_2 =\phi R_1$).  If $(BC)^2 = A$, we assume $BC=R_1$ and note that $\bar{R_2}BC = \phi \neq \varepsilon$.
Therefore, in the case $(BC)^2=A$,
\begin{align*}
T_1&=
 \displaystyle\sum_{R^2=A} 
\dfrac{g(\bar{A}B C) \, g(R) \, g(\bar{R}B) \, g(\bar{R}C) \, R(-1)}{g(\bar{A}B) g(\bar{A}C) g(A) g(\bar{R}BC)}
+
\dfrac{g(\bar{A}B C) \, g(R_1) \, q(q-1)}{g(B) g(C) g(\bar{A}B) g(\bar{A}C) g(A)}\\[4pt]
&=
 \dfrac{ g(\bar{A}BC)  \, g(\bar{R_2}B) \, g(\bar{R_2}C) \, g(R_2) \, g(R_2\bar{BC})}{ g(\bar{A}B) \, g(\bar{A}C) \,  g(A) BC(-1) \, q} \\[4pt]
& \qquad  +
\dfrac{ g(\bar{A}BC)  \, g(\bar{R_1}B) \, g(\bar{R_1}C) \, g(R_1) \, g(R_1\bar{BC})}{ g(\bar{A}B) \, g(\bar{A}C) \,  g(A) BC(-1) }
\times \left[1-
\dfrac{q(q-1) R_1(-1)}{g(B)\, g(\bar{B}) \, g(R_1 \bar{B}) \, g(\bar{R_1} B)}
\right].
\end{align*}
Applying  (\ref{for_GaussConj}) yields
\begin{equation*}
 1-\dfrac{q(q-1) R_1(-1)}{g(B)\, g(\bar{B}) \, g(R_1 \bar{B}) \, g(\bar{R_1} B)}
=
\begin{cases}
\frac{1}{q} & \textup{if $B\neq \varepsilon$ and $C\neq \varepsilon$},\\[4pt]
2-q & \textup{if $B= \varepsilon$ or $C = \varepsilon$ but not both},\\[4pt]
-q^2+q+1  & \textup{if $B= \varepsilon$ and $C = \varepsilon$}.
\end{cases}
\end{equation*}

\noindent Overall then
\begin{equation}\label{for_thm7_T1}
T_1=
\begin{cases}
0 & \textup{if $A \neq \square$,}\\[9pt]
\displaystyle\sum_{R^2=A} \dfrac{ g(\bar{A}BC)  \, g(\bar{R}B) \, g(\bar{R}C) \, g(R) \, g(R\bar{BC})}{ g(\bar{A}B) \, g(\bar{A}C) \,  g(A) BC(-1) \, q} 
& \genfrac{}{}{0pt}{0}{\textup{if $A = \square$ and $(BC)^2 \neq A$, } \textit{or} \phantom{AA}}{\textup{if $(BC)^2 = A$, $B \neq \varepsilon$ and $C \neq \varepsilon$,}}\\[21pt]
 -q+3  & \textup{if $(BC)^2 = A\neq \varepsilon$, $B$ or $C= \varepsilon$,}\\[18pt] 
 -q^2+2q+1 &  \genfrac{}{}{0pt}{0} {\textup{if $B^2=C^2 = A= \varepsilon$ and  $B \neq C$, } \textit{or}}{\textup{if $A=B=C= \varepsilon$.}\phantom{AAAAAAAAA\;\;}}\\[18pt]
\end{cases}
\end{equation}

\noindent Combining (\ref{for_thm7_T2}) and (\ref{for_thm7_T1}) yields Theorem \ref{thm_4} and Theorem \ref{cor_4} easily follows.

\end{proof}

\begin{proof}[Proof of Theorem \ref{cor_5}]

\noindent By Theorem \ref{thm_7}
\begin{multline}\label{for_thm5_thm7}
{_{4}F_3} \biggl( \begin{array}{cccc} A, & B, & C, & D \vspace{.02in}\\
\phantom{A} & A\bar{B}, & A\bar{C}, & A\bar{D}  \end{array}
\Big| \; {-1} \biggr)_{q}^{\star}
=
\frac{g(\bar{A}CD)}{g(C) g(D) g(\bar{A}C) g(\bar{A}D)} \times\\[4pt]
\frac{1}{q-1} \sum_{\psi \in \widehat{\mathbb{F}_q^{*}}} 
g(C \psi) g(D \psi) g(\bar{\psi}) g(\bar{A}\bar{\psi}) \;
{_{3}F_{2}} {\biggl( \begin{array}{ccc} A, & B, & \bar{\psi} \\
 \phantom{A,} & A\bar{B}, & A\psi \end{array}
\Big| \; 1 \biggr)}_{q}^{\star}\\[4pt]
+
\frac{q(q-1) CD(-1) \delta({\bar{A}CD})}{g(C) g(D) g(\bar{A}C) g(\bar{A}D)} \;
{_{2}F_{1}} {\biggl( \begin{array}{cc} A, & B \\
 \phantom{A,} & A\bar{B} \end{array}
\Big| \; {-1} \biggr)}_{q}^{\star}.
\end{multline}

\noindent By Theorems \ref{thm_4} and \ref{thm_3} we see that both terms on the right-hand side of (\ref{for_thm5_thm7}) equal zero if $A$ is not a square.
If $A \neq \varepsilon$ is a square, $B \neq \varepsilon$ and $B^2 \neq A$ then Theorem \ref{thm_4} tells us that, for all $\psi$,
\begin{equation}\label{for_thm5_thm4}
{_{3}F_{2}} {\biggl( \begin{array}{ccc} A, & B, & \bar{\psi} \\
 \phantom{A,} & A\bar{B}, & A\psi \end{array}
\Big| \; 1 \biggr)}_{q}^{\star}
=
\displaystyle\sum_{R^2=A} \dfrac{ g(\bar{A}B \bar{\psi})  \, g(\bar{R}B) \, g(\bar{R}\bar{\psi}) \, g(R) \, g(R\bar{B} \psi)}{ g(\bar{A}B) \, g(\bar{A\psi}) \,  g(A) B\bar{\psi}(-1) \, q} .
\end{equation}

\noindent Substituting (\ref{for_thm5_thm4}) into (\ref{for_thm5_thm7}) and rearranging yields the following theorem. 
\begin{theorem}\label{thm_5}
For $A$, $B$, $C$, $D \in \widehat{\mathbb{F}_q^{*}}$,
\begin{multline*}
{_{4}F_3} \biggl( \begin{array}{cccc} A, & B, & C, & D \vspace{.02in}\\
\phantom{A} & A\bar{B}, & A\bar{C}, & A\bar{D}  \end{array}
\Big| \; {-1} \biggr)_{q}^{\star}\\
=
\begin{cases}
0 & \textup{if $A \neq \square$,}\\[15pt]
\dfrac{g(\bar{A}) \, g(\bar{A}CD) }{ g(\bar{A}C) \, g( \bar{A}D)}
\displaystyle\sum_{R^2=A} 
{_{3}F_2} \biggl( \begin{array}{ccc} R\bar{B}, & C, & D \vspace{.02in}\\
\phantom{A} & R, & A\bar{B} \end{array}
\Big| \; 1 \biggr)_{q}^{\star}\\[21pt]
\quad+
\displaystyle\frac{q(q-1) \delta({\bar{A}CD})}{g(C) g(\bar{C}) g(A \bar{C}) g(\bar{A}C) }
{_{2}F_{1}} {\biggl( \begin{array}{cc} A, & B \\
 \phantom{A,} & A\bar{B} \end{array}
\Big| \; {-1} \biggr)}_{q}^{\star}
& \genfrac{}{}{0pt}{0}{\textup{if $A = \square$,  $A \neq \varepsilon$, } \phantom{AA}}{\textup{$B \neq \varepsilon$ and $B^2 \neq A$.}}
\end{cases}
\end{multline*}
\end{theorem}

\noindent Theorem \ref{cor_5} follows when we impose the additional condition $CD \neq A$ when $A$ is a square.

\end{proof}

\begin{remark}
We can also use Theorem \ref{thm_4} to evaluate (\ref{for_thm5_thm7}) when the conditions $A \neq \varepsilon$, $B \neq \varepsilon$ and $B^2 \neq A$ are not satisfied. 
In this case it will be necessary to consider certain values of $\psi$ separately. However the results are not as neat as the main cases. For example, if $A=\varepsilon$ and $B=\phi$ then
\begin{multline*}
{_{4}F_3} \biggl( \begin{array}{cccc} \varepsilon, & \phi, & C, & D \vspace{.02in}\\
\phantom{\varepsilon} & \phi, & \bar{C}, & \bar{D}  \end{array}
\Big| \; {-1} \biggr)_{q}^{\star}
=
-\dfrac{ g(CD) }{ g(C) \, g(D)}
\Biggl[ \displaystyle\sum_{R^2=\varepsilon} 
{_{3}F_2} \biggl( \begin{array}{ccc} R \phi, & C, & D \vspace{.02in}\\
\phantom{R \phi,} & R, & \phi \end{array}
\Big| \; 1 \biggr)_{q}^{\star}\\
+ (q-1) \left(1 + \displaystyle\frac{g(C\phi) \, g(D\phi) \, \phi(-1)}{g(C) \, g(D)}\right)\Biggr]
+ \displaystyle\frac{q(q-1) \delta({CD})}{g(C)^2 g(\bar{C})^2  }
(1 + \phi(-1)).
\end{multline*}
\end{remark}

\begin{proof}[Proof of Theorem \ref{cor_6}]
 \noindent By Theorem \ref{thm_7}
\begin{multline}\label{for_thm6_thm7}
{_{5}F_{4}} {\biggl( \begin{array}{ccccc} A, & B, & C, & D, & E \\
 \phantom{A} & A\bar{B}, & A\bar{C},  & A\bar{D}, & A\bar{E} \end{array}
\Big| \; 1 \biggr)}_{q}^{\star}
=
\frac{g(\bar{A}DE)}{g(D) g(E) g(\bar{A}D) g(\bar{A}E)} \times\\
\frac{1}{q-1} \sum_{\psi \in \widehat{\mathbb{F}_q^{*}}} 
g(D \psi) g(E \psi) g(\bar{\psi}) g(\bar{A}\bar{\psi}) \,
{_{4}F_{3}} {\biggl( \begin{array}{cccc} A, & B, & C, & \bar{\psi} \\
 \phantom{A,} & A\bar{B}, &  A\bar{C}, & A \psi \end{array}
\Big| \; {-1} \biggr)}_{q}^{\star}\\
+
\frac{q(q-1) DE(-1) \delta(\bar{A}DE)}{g(D) g(E) g(\bar{A}D) g(\bar{A}E)}\,
{_{3}F_{2}} {\biggl( \begin{array}{ccc} A, & B, &C \\
 \phantom{A,} & A\bar{B}, & A\bar{C} \end{array}
\Big| \; 1 \biggr)}_{q}^{\star}.
\end{multline}

\noindent By Theorems \ref{thm_5} and \ref{thm_4} we see that both terms on the right-hand side of (\ref{for_thm6_thm7}) equal zero if $A$ is not a square.
If $A \neq \varepsilon$ is a square, $B \neq \varepsilon$ and $B^2 \neq A$ then by Theorem \ref{thm_5}
\begin{multline}\label{for_thm6_thm5}
\frac{1}{q-1} \sum_{\psi \in \widehat{\mathbb{F}_q^{*}}} 
\displaystyle\frac{g(D \psi) g(E \psi) g(\bar{\psi}) g(\bar{A}\bar{\psi})}{g(D) g(E) } \,
{_{4}F_{3}} {\biggl( \begin{array}{cccc} A, & B, & C, & \bar{\psi} \\
 \phantom{A,} & A\bar{B}, &  A\bar{C}, & A \psi \end{array}
\Big| \; {-1} \biggr)}_{q}^{\star}\\
=
\frac{1}{q-1} \sum_{\psi \in \widehat{\mathbb{F}_q^{*}}} 
\dfrac{g(D \psi) g(E \psi) g(\bar{\psi}) g(\bar{A}) \, g(\bar{A}C \bar{\psi})}{g(D) g(E) g(\bar{A}C)} \,
\displaystyle\sum_{R^2=A} 
{_{3}F_2} \biggl( \begin{array}{ccc} R\bar{B}, & C, & \bar{\psi} \vspace{.02in}\\
\phantom{A} & R, & A\bar{B} \end{array}
\Big| \; 1 \biggr)_{q}^{\star}\\
+
\displaystyle\frac{q \, g( \bar{A} CD) g( \bar{A} CE ) }{g(C) g(D) g(E)  g(\bar{A}C)} \,
{_{2}F_{1}} {\biggl( \begin{array}{cc} A, & B \\
 \phantom{A,} & A\bar{B} \end{array}
\Big| \; {-1} \biggr)}_{q}^{\star}.
\end{multline}

\noindent For brevity we will refer to the first term on the right-hand side of (\ref{for_thm6_thm5}) as $T$. We expand the $_3F_2$ by definition to get
\begin{multline*}
T = 
\dfrac{g(\bar{A}) }{g(D) g(E) g(\bar{A}C)} \displaystyle\sum_{R^2=A} 
\dfrac{1}{(q-1)^2}
\sum_{\chi \in \widehat{\mathbb{F}_q^{*}}} 
\dfrac{g(R \bar{B} \chi) g(C \chi) g(\bar{R \chi}) g(\bar{A}B \bar{\chi}) g(\bar{\chi}) \chi(-1)}{g(R \bar{B}) g(C) g(\bar{R}) g(\bar{A}B)}\\
\times
\sum_{\psi \in \widehat{\mathbb{F}_q^{*}}} 
g(D \psi) g(E \psi) g(\bar{A}C \bar{\psi}) g(\chi \bar{\psi}).
\end{multline*}
We now let $\psi \to \chi\psi$ to get
\begin{multline*}
T = 
\dfrac{g(\bar{A}) }{g(D) g(E) g(\bar{A}C)} \displaystyle\sum_{R^2=A} 
\dfrac{1}{(q-1)^2}
\sum_{\chi \in \widehat{\mathbb{F}_q^{*}}} 
\dfrac{g(R \bar{B} \chi) g(C \chi) g(\bar{R \chi}) g(\bar{A}B \bar{\chi}) g(\bar{\chi}) \chi(-1)}{g(R \bar{B}) g(C) g(\bar{R}) g(\bar{A}B)}\\
\times
\sum_{\psi \in \widehat{\mathbb{F}_q^{*}}} 
g(D \chi \psi) g(E\chi  \psi) g(\bar{A}C \bar{\chi \psi}) g( \bar{\psi}).
\end{multline*}
Now if $A \neq CD$ and $A \neq CE$ then by Theorem \ref{thm_HP}
\begin{equation*}
 \frac{1}{q-1} \sum_{\psi \in \widehat{\mathbb{F}_q^{*}}} 
g(D \chi \psi) g(E\chi  \psi) g(\bar{A}C \bar{\chi \psi}) g( \bar{\psi})
=
\frac{ g(D \chi) g(E \chi) g(\bar{A}CD) g(\bar{A}CE) g(A \bar{CDE \chi})}{q \, A \bar{CDE \chi}(-1)}.
\end{equation*}
Therefore, if $A \neq CD$ and $A \neq CE$,
\begin{equation*}
T=
 \dfrac{g(\bar{A}) g(\bar{A}CD) g(\bar{A}CE) g(A \bar{CDE})}{ g(\bar{A}C) \, q \, A \bar{CDE}(-1)}
\displaystyle\sum_{R^2=A} 
{_{4}F_3} \biggl( \begin{array}{cccc} R\bar{B}, & C, & D,  & E  \vspace{.02in}\\
 \phantom{A,} & R, & \bar{A}CDE, & A\bar{B} \end{array}
 \Big| \; 1 \biggr)_{q}^{\star}.
\end{equation*}
Overall then we have proved the following.
\begin{theorem}\label{thm_6}
For $A$, $B$, $C$, $D$, $E \in \widehat{\mathbb{F}_q^{*}}$, such that, when $A$ is a square, $A \neq \varepsilon$, $B \neq \varepsilon$, $B^2 \neq A$, $A\neq CD$, $A\neq CE$,
 \begin{multline*}
{_{5}F_4} \biggl( \begin{array}{ccccc} A, & B, & C, & D, & E \vspace{.02in}\\
\phantom{A} & A\bar{B}, & A\bar{C}, & A\bar{D}, & A\bar{E}  \end{array}
\Big| \; 1 \biggr)_{q}^{\star}\\
=
\begin{cases}
0 & \textup{if $A \neq \square$,}\\[18pt] 
 \dfrac{g(\bar{A}) g(\bar{A}DE) g(\bar{A}CD) g(\bar{A}CE) g(A \bar{CDE})}{ g(\bar{A}C) g(\bar{A}D) g(\bar{A}E) \, q \, A \bar{CDE}(-1)}\\[18pt]
\quad \times
\displaystyle\sum_{R^2=A} 
{_{4}F_3} \biggl( \begin{array}{cccc} R\bar{B}, & C, & D,  & E  \vspace{.02in}\\
 \phantom{A,} & R, & \bar{A}CDE, & A\bar{B} \end{array}
 \Big| \; 1 \biggr)_{q}^{\star}\\[24pt]
 + \displaystyle\frac{g(\bar{A}DE) g(\bar{A}CD) g(\bar{A}CE) \, q}{g(C) g(D) g(E)  g(\bar{A}C) g(\bar{A}D) g(\bar{A}E)}\,
{_{2}F_1} \biggl(  \begin{array}{ccccc} A, & B \vspace{.02in}\\
 \phantom{A,} & A\bar{B} \end{array}
 \Big| \; {-1} \biggr)_{q}^{\star}\\[24pt]
 +
\dfrac{q(q-1) \delta(\bar{A}DE)}{g(D) g(E) g(\bar{E}) g(\bar{D})}\,
{_{3}F_{2}} {\biggl( \begin{array}{ccc} A, & B, &C \\
 \phantom{A,} & A\bar{B}, & A\bar{C} \end{array}
\Big| \; 1 \biggr)}_{q}^{\star}
& \textup{otherwise.}
 \end{cases}
 \end{multline*}
\end{theorem}

\noindent Theorem \ref{cor_6} follows when $A\neq DE$ and $A \neq CDE$ in the case $A$ is a square.
\end{proof}

\begin{proof}[Proof of Theorem \ref{cor_7}]
We start by considering the case when $n=0$. By definition
\begin{equation*}
{_{1}F_{0}} {\biggl( \begin{array}{c} A_0 \\
 \phantom{A_0}  \end{array}
\Big| \; 1 \biggr)}_{q^{\star}}
=
\frac{1}{(q-1) \, g(A_0)}  \sum_{\chi \in \widehat{\mathbb{F}_q^{*}}} g(A_0 \chi) g(\bar{\chi}) \chi(-1), 
\end{equation*}
which equals zero by Proposition \ref{thm_2}. The cases $n=1,2,3,4$ have all been dealt with in Theorems \ref{thm_3}, \ref{thm_4}, \ref{thm_5} and  \ref{thm_6} respectively. The rest follow by using induction on $n$ in Theorem \ref{thm_7}.
\end{proof}


\section{Concluding Remarks}\label{sec_cr}
The methods used in this paper for obtaining transformations broadly mirror the methods used by Whipple in proving his results and, 
just as in the classical case, these methods break down if we try to extend to results involving a well-poised ${_{6}F_{5}}(\cdots)^{\star}$ in its most general form (and when $A$ is a square). 
However, as noted in Section \ref{sec_Intro}, Whipple does obtain transformations for well-poised ${_6F_5}[\cdots | -1]$ and ${_7F_6}[\cdots | 1]$ where the `$b$' parameter is specialized to equal $1+\frac{1}{2}a$. 
Therefore such a series has $1+\frac{1}{2}a$ as one of its numerator parameters with the corresponding denominator parameter of $\frac{1}{2}a$. 
These are obviously different values but their finite field analogues would be the same character, as the analogue of the $1$ would be the trivial character. This leads to problems in trying to produce analogous results in the finite field case at these higher orders. 
It is not clear if a different interpretation of the series in this case may lead to a more appropriate finite field analogue for these values.


\end{document}